\documentclass[12pt, reqno]{amsart}
\usepackage{amsmath, amsthm, amscd, amsfonts, amssymb, graphicx, color}
\usepackage[bookmarksnumbered, colorlinks, plainpages]{hyperref}
\input{mathrsfs.sty}

\newtheorem{theorem}{Theorem}[section]

\newtheorem{corollary}[theorem]{Corollary}
\theoremstyle{definition}
\newtheorem{definition}[theorem]{Definition}
\newtheorem{example}[theorem]{Example}

\theoremstyle{remark}
\newtheorem{remark}[theorem]{Remark}

\numberwithin{equation}{section}

\begin{document}

\title{Quasi-representations of Finsler modules over $C^*$-algebras}

\author[M. Amyari, M. Chakoshi, M. S. Moslehian]{M. Amyari$^1$$^*$, M. Chakoshi $^1$ and M. S. Moslehian$^2$}
\address{$^1$ Department of Mathematics, Faculty of
science, Islamic Azad University-Mashhad Branch, Mashhad 91735,
Iran.} \email{amyari@mshdiau.ac.ir and $Maryam_-$Amyari@yahoo.com}
\email{m-chakoshi@mshdiau.ac.ir }
\address{$^2$ Department of Pure Mathematics, Center of Excellence in Analysis on Algebraic Structures (CEAAS), Ferdowsi University of
Mashhad, P.O. Box 1159, Mashhad 91775, Iran.}
\email{moslehian@ferdowsi.um.ac.ir and moslehian@member.ams.org}


\keywords{Finsler module, quasi-representation, $\varphi$-morphism,
nondegenerate quasi-representation, irreducible
quasi-representation.\\
$^*$ Corresponding author}

\begin{abstract}
We show that every Finsler module over a $C^*$-algebra has a
quasi-representation into the Banach space
$\mathbb{B}(\mathscr{H},\mathscr{K})$ of all bounded linear
operators between some Hilbert spaces $\mathscr{H}$ and
$\mathscr{K}$. We define the notion of completely positive
$\varphi$-morphism and establish a Stinespring type theorem in the
framework of Finsler modules over $C^*$-algebras. We also
investigate the nondegeneracy and the irreducibility of
quasi-representations.
\end{abstract}

\maketitle

\section{INTRODUCTION}

The notation of Finsler module is an interesting generalization of
that of Hilbert $C^*$-module. It is a useful tool in the operator
theory and the theory of operator algebras and may be served as a
noncommutative version of the concept of Banach bundle, which is an
essential concept in the Finsler geometry. In 1995 Phillips and
Weaver \cite{ph} showed that if a $C^*$-algebra $\mathscr{A}$ has no
nonzero commutative ideal, then any Finsler $\mathscr{A}$-module
must be a Hilbert $C^*$-module. If $\mathscr{A}$ is the commutative
$C^*$-algebra $C_0(X)$ of all continuous complex-valued functions
vanishing at infinity on a locally compact Hausdorff space $X$, then
any Finsler $\mathscr{A}$-module is isomorphic to the module of
continuous sections of a bundle of Banach spaces over $X$. The
concept of a $\varphi$-morphism between Finsler modules was
introduced in \cite{M}.

The Gelfand--Naimark--Segal (GNS) representation theorem is one of
the most useful theorems, which is applied in operator algebras and
mathematical physics. That provides a procedure to construct
representations of $C^*$-algebras. A generalization of GNS
construction to a topological $*$-algebra established by Borchers,
Uhlmann and Powers leading to unbounded $*$-representations of
$*$-algebras; see \cite{po}. Another is a generalization of a
positive linear functional to a completely positive map studied by
Stinespring \cite{st}, see also \cite{JOI1}.

Let $\mathscr{A}$ be a $C^*$-algebra and let $\mathscr{A}^+$ denote
the positive cone of all positive elements of $\mathscr{A}$. We
define a Finsler $\mathscr{A}$-module to be a right
$\mathscr{A}$-module $\mathscr{E}$ equipped with a map
$\rho: \mathscr{E} \to \mathscr{A}^+$ (denoted by $\rho_\mathscr{A}$ if there is an ambiguity) satisfying the following conditions:\\
(i) The map $\|\cdot\|_\mathscr{E}: x\longmapsto \|\rho(x)\|$ makes $\mathscr{E}$ into a Banach space.\\
(ii) $\rho(xa)^2=a^*\rho(x)^2 a$, for all $a \in \mathscr{A}$ and $x \in \mathscr{E}.$\\

A Finsler module $\mathscr{E}$ over a $C^*$-algebra $\mathscr{A}$ is
said to be full if the linear span of $\{\rho(x)^2: x\in
\mathscr{E}\}$ is dense in $\mathscr{A}$. For example, if
$\mathscr{E}$ is a (full) Hilbert $C^*$-module over $\mathscr{A}$
(see \cite{MT}), then $\mathscr{E}$ together with $\rho (x) =
<x,x>^{\frac{1}{2}}$ is a (full) Finsler module over $\mathscr{A}$,
since
\begin{eqnarray*}
\rho(xa)^2 = <xa,xa>=a^*<x,x>a =a^*\rho(x)^2 a.
\end{eqnarray*}
In particular, every $C^*$-algebra $\mathscr{A}$ is a full Finsler
module over $\mathscr{A}$ under the mapping $\rho(x)=(x^*x
)^{\frac{1}{2}}$.

Our goal is to extend the notion of a representation of a Hilbert
$C^*$-module to the framework of Finsler $\mathscr{A}$-modules. We
show that every Finsler $\mathscr{A}$-module has a
quasi-representation into the Banach space
$\mathbb{B}(\mathscr{H},\mathscr{K})$ of all bounded linear
operators between some Hilbert spaces $\mathscr{H}$ and
$\mathscr{K}$. We define the notion of completely positive
$\varphi$-morphism and establish a Stinespring type theorem in the
framework of Finsler modules over $C^*$-algebras. We also introduce
the notions of the nondegeneracy and the irreducibility of
quasi-representations and study some interrelations between them.

\section{QUASI-REPRESENTATIONS OF FINSLER MODULES}

We start our work by giving the definition of a $\varphi$-morphism
of a Finsler module.

\begin{definition} Suppose that $(\mathscr{E}, \rho_{\mathscr{A}} )$ and $(\mathscr{F}, \rho_{\mathscr{B}})$ are Finsler modules over $C^*$-algebras $\mathscr{A}$ and $\mathscr{B}$, respectively and $\varphi: \mathscr{A}\to \mathscr{B}$ is a $*$-homomorphism of $C^*$-algebras. A (not necessarily linear) map $\Phi: \mathscr{E} \to \mathscr{F}$ is said to be a $\varphi$-morphism of Finsler modules if the following conditions are satisfied:

(i) $\rho_\mathscr{B} (\Phi(x))=\varphi(\rho_{\mathscr{A}} (x))$;

(ii) $\Phi(xa)=\Phi(x)\varphi(a)$.\\
for all $x\in \mathscr{E}$ and $a\in \mathscr{A}$.  In the case of
Hilbert $C^*$-modules, $\Phi$ is assumed to be linear and then
condition (ii) is deduced from (i).
\end{definition}
\noindent Now we introduce the notion of a quasi-representation of a
Finsler module. Due to $\mathbb{B}(\mathscr{H},\mathscr{K})$ is a
Hilbert $C^*$-module over $\mathbb{B}(\mathscr{H})$ via $\langle T,S
\rangle=T^*S$, we can endow $\mathbb{B}(\mathscr{H},\mathscr{K})$ a
Finsler structure by
\begin{equation}
\rho_0(T)=(T^*T)^{\frac{1}{2}}.\label{10}
\end{equation}

\begin{definition}
Let $(\mathscr{E}, \rho)$ be a Finsler module over a $C^*$-algebra
$\mathscr{A}$. A map $\Phi: \mathscr{E} \to
\mathbb{B}(\mathscr{H},\mathscr{K})$, where $\varphi: \mathscr{A}\to
\mathbb{B}(\mathscr{H})$ is a representation of $\mathscr{A}$ is
called a quasi-representation of $\mathscr{E}$ if
$\rho_0(\Phi(x))=\varphi(\rho(x))$ for all $x\in \mathscr{E}$.
\end{definition}

We are going to show that for every Finsler $\mathscr{A}$-module
there is a quasi-representation to
$\mathbb{B}(\mathscr{H},\mathscr{K})$ for some Hilbert spaces
$\mathscr{H}$ and $\mathscr{K}$, see also \cite{mor}.

\begin{theorem} \label{main}
Suppose $\mathscr{E}$ is a Finsler $\mathscr{A}$-module with the
associated map $\rho: \mathscr{E}\to \mathscr{A}^+$. Then there is a
quasi-representation $\Phi: \mathscr{E}\to
\mathbb{B}(\mathscr{H},\mathscr{K})$ for some Hilbert spaces
$\mathscr{H}$ and $\mathscr{K}$.\label{3}
\end{theorem}

\begin{proof} By the Gelfand--Naimark theorem for $C^*$-algebras, there is a representation $\varphi: \mathscr{A}\to \mathbb{B}(\mathscr{H})$ for some Hilbert space $\mathscr{H}$. We want to construct a Hilbert space $\mathscr{K}$. Put
\begin{eqnarray*}
\mathscr{K}_0:= {\rm span}\{\varphi(a)f: a\in\mathscr{A}, f:
\mathscr{E}\to \mathscr{H} {\rm ~is~ a~ map~ with~ a~ finite~
support} \}
\end{eqnarray*}
and define on $\mathscr{K}_0$ an inner product by
\begin{eqnarray*}
\langle \varphi(a)f, \varphi(b)g \rangle = \sum_{x\in \mathscr{E}}
\langle \varphi(a)f(x), \varphi(b)g(x) \rangle.
\end{eqnarray*}
Note that if $\langle \sum_{i=1}^{n}\varphi(a_i)f_i,
\sum_{i=1}^{n}\varphi(a_i)f_i \rangle=0$, then
$$\sum_{x\in \mathscr{E}} \langle \sum_{i=1}^{n}\varphi(a_i)f_i (x), \sum_{i=1}^{n}\varphi(a_i)f_i (x)\rangle=0\,.$$ Thus $\sum_{i=1}^{n}\varphi(a_i)f_i (x)=0$ for each $x\in \mathscr{E}$, whence $\sum_{i=1}^{n}\varphi(a_i)f_i=0$.\\
Let us consider the closure $\overline{\mathscr{K}_0}$ of
$\mathscr{K}_0$ to get a Hilbert space, which is denoted by
$\mathscr{K}.$ For any $y\in \mathscr{E}$ and $h\in \mathscr{H}$,
the map $h_y: \mathscr{E}\to \mathscr{H}$ defined by
\begin{eqnarray*}
h_y(x)=\begin{cases}
h& x=y \\
0& x\neq y
\end{cases}
\end{eqnarray*}
has a finite support. For $x\in \mathscr{E}$, define $\Phi(x):
\mathscr{H} \to \mathscr{K}$ by $\Phi(x)h=\varphi(\rho(x))h_x$. We
show that $\Phi(x)\in \mathbb{B}(\mathscr{H},\mathscr{K})$. Clearly
$\Phi(x)$ is linear. Also $\Phi(x)$ is bounded, since
\begin{eqnarray*}
\|\Phi(x)h\|^2&=&\langle\Phi(x)h, \Phi(x)h\rangle=\langle
\varphi(\rho(x))h_x, \varphi(\rho(x))h_x\rangle\\&=&
\sum_{y\in \mathscr{E}}\langle \varphi(\rho(x))h_x(y), \varphi(\rho(x))h_x(y)\rangle=\langle\varphi(\rho(x))h, \varphi(\rho(x))h\rangle\\
&\leq& \|\varphi(\rho(x))\|^2\| h \|^2\,,
\end{eqnarray*}
whence $\| \Phi(x)\|\leq\|\varphi(\rho(x))\|$.\\
Further,
\begin{eqnarray*}
\langle\Phi(x)^*\Phi(x)h, h'\rangle &=& \langle\Phi(x)h~,\Phi(x)h'\rangle=\langle\varphi(\rho(x))h_x, \varphi(\rho(x))h'_x\rangle\\
&=&\sum_{y\in \mathscr{E}}\langle\varphi(\rho(x)) h_x(y), \varphi(\rho(x))h'_x(y)\rangle\\
&=& \langle \varphi(\rho(x))h, \varphi(\rho(x))h'\rangle=\langle
\varphi(\rho(x)^2) h, h'\rangle,
\end{eqnarray*}
for all $h,h'\in \mathscr{H}$ and $x\in \mathscr{E}$. Hence
$\Phi(x)^*\Phi(x)=\varphi(\rho(x)^2)$. Hence
\begin{equation}
(\Phi(x)^*\Phi(x))^{\frac{1}{2}}=\varphi(\rho(x)).\label{4}
\end{equation}

It follows from (\ref{10}) and equality (\ref{4}) that
$\rho_0(\Phi(x))=\varphi(\rho(x))$.
\end{proof}

\begin{remark}
If $\Phi$ is surjective and $\mathbb{B}(\mathscr{H},\mathscr{K})$ is
a full Finsler $\mathbb{B}(\mathscr{H})$-module, then by
\cite[Theorem 3.4(iv)]{M}, $\varphi$ is surjective.
\end{remark}

In the next section the notion of completely positive
$\varphi$-morphism is introduced and a construction of Stinespring's
theorem for Finsler modules is given.

\section{A STINESPRING TYPE THEOREM FOR FINSLER MODULES}

The Stinespring theorem was first introduced in the work of
Stinespring in 1995 that described the structure of completely
positive maps of a $C^*$-algebra into the $C^*$-algebra of all
bounded linear operators on a Hilbert space; see \cite{st}. Recently
Asadi \cite{As} proved this theorem for Hilbert $C^*$-modules.
Further, Bhat et al. \cite{Bha} improve the result of \cite{As} with
omitting a technical condition. In this section we intend to
establish a Stinespring type theorem in the framework of Finsler
modules over $C^*$-algebras.

\noindent A $\varphi$-morphism $\Phi: \mathscr{E} \to
\mathbb{B}(\mathscr{H},\mathscr{K})$ is called completely positive
if the map $\varphi: \mathscr{A} \to \mathbb{B} (\mathscr{H})$ is
completely positive.

\begin{theorem}
Let $(\mathscr{E}, \rho)$ be a Finsler module over a unital
$C^*$-algebra $\mathscr{A}$, let $\mathscr{H}, \mathscr{K}$ be
Hilbert spaces and let $\Phi: \mathscr{E} \to
\mathbb{B}(\mathscr{H},\mathscr{K})$ be a completely positive map
associated to a completely positive map $\varphi: \mathscr{A} \to
\mathbb{B} (\mathscr{H})$. Then there exist Hilbert spaces
$\mathscr{H}', \mathscr{K}'$ and isometries $V: \mathscr{H} \to
\mathscr{H}', W: \mathscr{K} \to \mathscr{K'}$, a $*$-homomorphism
$\theta: \mathscr{A} \to \mathbb{B} (\mathscr{H}')$ and a
$\theta$-morphism $\Psi: \mathscr{E} \to
\mathbb{B}(\mathscr{H}',\mathscr{K}')$ such that $\varphi(a)=V^*
\theta(a) V, \Phi(x)=W^* \Psi(x) V$ for all $x\in \mathscr{E}$ and
$a\in \mathscr{A}$.
\end{theorem}

\begin{proof} By \cite[Theorem 4.1]{paul} there exist a Hilbert space $\mathscr{H}'=\mathscr{A}\otimes \mathscr{H}$, a representation $\theta: \mathscr{A} \to \mathbb{B}(\mathscr{H'})$ and an isometry $V: \mathscr{H} \to \mathscr{H}'$ defined by $V(h)=1\otimes h$ such that $\varphi(a)=V^* \theta(a) V$. We may consider a minimal Stinespring representation for $\theta$, where $\mathscr{H}'$ is the closed linear span of $\{\theta(a)Vh: a\in \mathscr{A}, h\in \mathscr{H}\}$.\\
Now, we put $\mathscr{K}'$ to be the closed linear span of
$\{\Phi(x)h: x\in \mathscr{E}, h\in \mathscr{H}\}$ and define the
mapping $\Psi: \mathscr{E} \to
\mathbb{B}(\mathscr{H}',\mathscr{K}'), x\mapsto \Psi(x)$,
where $\Psi(x): {\rm span} \{\theta(a)Vh, a\in \mathscr{A}, h\in \mathscr{H}\} \to \mathscr{K}'$ is defined by $\Psi(x) (\sum_{i=1}^{n} \theta(a_i)Vh_i) = \sum_{i=1}^{n} \Phi(xa_i)h_i$ for $x \in \mathscr{E}, a_i\in \mathscr{A}, h_i \in \mathscr{H}.$\\
The map $\Psi(x)$ is well-defined and bounded, since
\begin{eqnarray*}
\left\| \Psi(x) (\sum_{i=1}^{n} \theta(a_i) V h_i) \right\|^2 & =&
\left\| \sum_{i=1}^{n} \Phi(x a_i) h_i\right\|^2\\
&=& \sum_{i,j=1}^{n}\langle \Phi(x a_j)^* \Phi(x a_i)h_i,
h_j\rangle\\
&=& \sum_{i,j=1}^{n}\langle \varphi(a_j^*)\Phi(x)^*\Phi(x) \varphi(a_i) h_i, h_j \rangle\\
&=& \sum_{i,j=1}^{n}\langle \varphi(a_j^*) \varphi(\rho(x)^2)
\varphi(a_i )h_i, h_j \rangle\\
&=& \sum_{i,j=1}^{n}\langle \varphi(a_j^*\rho(x)^2 a_i )h_i, h_j \rangle\\
&=& \sum_{i,j=1}^{n}\langle V^*\theta(a_j^*\rho(x)^2 a_i ) V h_i,
h_j \rangle\\
&=&\sum_{i,j=1}^{n}\langle \theta(\rho(x)^2 ) \theta(a_i) V h_i, \theta(a_j) V h_j \rangle\\
&\leq & \left\| \theta(\rho(x)^2) \right\| \left\| \sum_{i=1}^{n}
\theta(a_i)V h_i \right\|^2\\
&\leq& \left\| \rho(x) \right\|^2 \left\| \sum_{i=1}^{n}\theta(a_i) V h_i\right\|^2 \\
&=& \left\| x \right\|^2 \left\| \sum_{i=1}^{n} \theta(a_i) V
h_i\right\|^2.
\end{eqnarray*}
The mapping $\Psi$ is a $\theta$-morphism, since for all $a,b\in
\mathscr{A}$ and $h,g\in \mathscr{H}$
\begin{eqnarray*}
\langle\Psi(x)^*\Psi(x) (\theta (a) V h), \theta(b) V g\rangle &=&
\langle\Psi(x) (\theta (a) V h), \Psi(x)(\theta(b) V g)\rangle\\
&=& \langle \Phi(xa) h, \Phi(xb)g\rangle\\
&=& \langle \Phi(x)
\varphi(a) h, \Phi(x)\varphi(b) g\rangle \\
&=& \langle
\Phi(x)^*\Phi(x) \varphi(a) h, \varphi(b) g\rangle\\
&=& \langle
\varphi(\rho(x)^2) \varphi(a) h, \varphi(b) g\rangle\\
&=& \langle
\varphi(b^*\rho(x)^2 a) h, g\rangle\\
&=& \langle V^*\theta(b^*\rho(x)^2 a) Vh, g\rangle
\\&=& \langle
\theta(\rho(x)^2 )\theta(a) Vh, \theta(b) V g\rangle,
\end{eqnarray*}
whence $\Psi(x)^*\Psi(x)=\theta(\rho(x)^2 )$. Moreover
\begin{eqnarray*}
\Psi(x)\theta(a)(\theta(b)Vh)&=&\Psi(x)(\theta(ab)Vh)\\&=&\Phi(x(ab))h\\&=&\Phi((xa)b)h\\&=&\Psi(xa)(\theta(b)Vh),
\end{eqnarray*}
so that $\Psi(x)\theta(a)=\Psi(xa)$.\\
Since $\mathscr{K}'\subseteq \mathscr{K}$ we can consider a map $W$ as the orthogonal projection of $\mathscr{K}$ onto $\mathscr{K}'$. Hence $W^*: \mathscr{K}'\to \mathscr{K}$ is the inclusion map, whence for any $k'\in \mathscr{K}'$ we have $WW^*(k')=W(k')=k'$, that is $WW^*=I_{\mathscr{K}'}$.\\
Finally we observe that $W^*\Psi(x)Vh=\Psi(x) V
h=\Psi(x)(\theta(1)Vh)=\Phi(x)h$, that is $W^*\Psi(x)V=\Phi(x)$.
\end{proof}

\section{NONDEGENERATE AND IRREDUCIBLE QUASI-REPRESENTATIONS}

In this section we define the notions of nondegenerate and
irreducible quasi-representations of Finsler modules and describe
relations between the nondegeneracy and the irreduciblity, see
\cite{lji}. Throughout this section we assume that the
quasi-representations satisfy condition (ii) of Definition 2.1.

\begin{definition} Let $\Phi: \mathscr{E}\to \mathbb{B}(\mathscr{H},\mathscr{K})$ be a quasi-representation of a Finsler module $\mathscr{E}$ over a $C^*$-algebra $\mathscr{A}$. The map $\Phi$ is said
to be nondegenerate if
$\overline{\Phi(\mathscr{E})\mathscr{H}}=\mathscr{K}$ and
$\overline{\Phi(\mathscr{E})^{*}\mathscr{K}}=\mathscr{H}$ (or
equivalently, if there exist $\xi\in \mathscr{H},\eta\in
\mathscr{K}$ such that $\Phi(\mathscr{E})\xi=0$ and
$\Phi(\mathscr{E})^{*}\eta=0$, then $\xi=\eta=0)$. Recall that a
representation $\varphi: \mathscr{A}\to \mathbb{B}(\mathscr{H})$ of
a $C^*$-algebra $\mathscr{A}$ is nondegenerate if
$\overline{\varphi(\mathscr{A})\mathscr{H}}=\mathscr{H}$ (or
equivalently, if there exists $\xi\in \mathscr{H}$ such that
$\varphi(\mathscr{A})\xi=0$, then $\xi=0)$, see \cite[definition
A.1.]{ske}.
\end{definition}

\begin{theorem} If $\Phi: \mathscr{E}\to \mathbb{B}(\mathscr{H},\mathscr{K})$ is a nondegenerate quasi-representation, then $\varphi: \mathscr{A}\to \mathbb{B}(\mathscr{H})$ is a nondegenerate representation. If $\mathscr{E}$ is full and $\varphi$ is nondegenerate, then $\Phi$ is also nondegenerate.\label{5}
\end{theorem}

\begin{proof} Suppose that $\Phi$ is nondegenerate and $\varphi(\mathscr A)\xi=0$. It follows from the Hewitt--Cohen factorization theorem that
$\Phi(\mathscr{E})\xi=\Phi(\mathscr{E}\mathscr
A)\xi=\Phi(\mathscr{E})\varphi(\mathscr A)\xi=0$. We conclude that
$\xi=0$. Thus $\varphi$ is nondegenerate.

Suppose that $\Phi(\mathscr{E})\xi=0$ for some $\xi\in \mathscr{H}$.
Then for any $x\in \mathscr{E}$ we have $\| \Phi(x)\xi\|^{2}=\langle
\Phi(x)^* \Phi(x)\xi, \xi\rangle=\langle\varphi(\rho(x)^2)\xi,
\xi\rangle=\| \varphi(\rho(x)) \|^2=0$. Since $\mathscr{E}$ is a
full Finsler $\mathscr A$-module, $a={\displaystyle{\lim_{n\to
\infty}}} \sum_{i=1}^{k_n} \lambda_{i,n} \rho(x_{i,n})^2$ for some
$k_n \in \mathbb{N}$, $x_{i,n}\in \mathscr{E} $ and
$\lambda_{i,n}\in \Bbb{C}$. Hence
\begin{eqnarray*}
\varphi(a)\xi={\displaystyle{\lim_{n\to \infty}}} \sum_{i=1}^{k_n}
\lambda_{i,n}\varphi(\rho(x_{i,n}))^2\xi={\displaystyle{\lim_{n\to
\infty}}} \sum_{i=1}^{k_n}
\lambda_{i,n}\varphi(\rho(x_{i,n}))\varphi(\rho(x_{i,n}))\xi=0,
\end{eqnarray*}
whence $\xi=0$.
\end{proof}

\begin{remark} The second result of Theorem \ref{5} may fail, if the condition of being full is dropped.
To see this take $\mathscr A$ to be a nondegenerate von Neumann
algebra acting on a Hilbert space,
which has a nontrivial central projection $P$. Hence the identity map $\varphi: \mathscr A \to \mathbb{B}(\mathscr{H})$ is assumed to be nondegenerate.\\
Put $\mathscr{E}=\mathscr A P=\{aP: a\in \mathscr A\}$ as a Finsler
$\mathscr A$-module equipped with $\rho(aP)=\vert aP\vert $. Clearly
$\mathscr A P$ is not full. The identity map $\Phi: \mathscr A P \to
\mathbb B(\mathscr{H})$ satisfies the following:

(i) $\rho_0\Phi(aP)=\rho_0(aP)=\vert aP\vert=\varphi(\vert
aP\vert)=\varphi\rho(aP)$, where $\rho_0$ is defined as in
(\ref{10}).

(ii) $\Phi(aPb)=\Phi(aP)\varphi(b)$ for all $b\in \mathscr A$. \\
Hence $\Phi$ is a quasi-representation of $\mathscr E$, which is not
clearly nondegenerate, since
\begin{eqnarray*}
\overline{\Phi(\mathscr E)\mathscr H}= \overline{\mathscr A
P(\mathscr{H})}=\overline{P(\mathscr A \mathscr{H})}\subseteq
\overline{P(\mathscr{H})}=P(\mathscr{H}) \neq
\mathscr{H}.\end{eqnarray*}\label{7}
\end{remark}

In the following corollary we investigate a condition under which
the representation $\varphi$ and the quasi-representation $\Phi$ are
nondegenerate.
\begin{corollary} If $\varphi(\rho(x))=I_\mathscr{H}$, then both $\Phi$ and $\varphi$ are nondegenerate.
\end{corollary}
\begin{proof} Suppose $\Phi(\mathscr{E})\xi=0$ for some $\xi\in \mathscr{H}$. Then for all $x\in \mathscr{E}$
we have $ \| \Phi(x)\xi\|^2=\langle \Phi(x)^*
\Phi(x)\xi~,~\xi\rangle=\langle \varphi(\rho(x)^2)\xi,
\xi\rangle=\|\xi\|^2=0$, so that $\xi=0$. The nondegeneracy of
$\varphi$ follows from Theorem \ref{5}.
\end{proof}

\begin{definition} Let $\Phi: \mathscr{E}\to B(\mathscr{H},\mathscr{H}')$ be a
quasi-representation of a Finsler module $\mathscr{E}$ over a
$C^*$-algebra $\mathscr{A}$ and let $\mathscr{K},\mathscr{K}'$ be
closed subspaces of $\mathscr{H}$ and $\mathscr{H}'$, respectively.
A pair of subspaces $(\mathscr{K},\mathscr{K}')$ is said to be
$\Phi$-invariant if $\Phi(\mathscr{E})\mathscr{K}\subseteq
\mathscr{K}'$ and $\Phi(\mathscr{E})^{*}\mathscr{K}'\subseteq
\mathscr{K}$. The quasi-representation $\Phi$ is said to be
irreducible if $(0,0)$ and $(\mathscr{H},\mathscr{H}')$ are the only
$\Phi$-invariant pairs. Recall that a representation $\varphi:
\mathscr{A}\to \mathbb{B}(\mathscr{H})$ of a $C^*$-algebra
$\mathscr{A}$ is irreducible if $0$ and $\mathscr{H}$ are only
closed subspaces of $\mathscr{H}$ being $\varphi$-invariant, i.e.
are invariant for $\varphi(\mathscr{A})$.
\end{definition}

\begin{theorem}
Suppose that the quasi-representation $\Phi: \mathscr{E}\to
\mathbb{B}(\mathscr{H},\mathscr{K})$ constructed in Theorem
\ref{main} is irreducible. Then so is $\varphi: \mathscr{A}\to
\mathbb{B}(\mathscr{H})$. If $\mathscr{E}$ is full and $\varphi$ is
irreducible, then $\Phi$ is irreducible.\label{6}
\end{theorem}

\begin{proof} Suppose that $\Phi$ is irreducible and a closed subspace $\mathscr{K}$ of $\mathscr{H}$ is $\varphi$-invariant.
Consider $\mathscr{K}'=\overline{\Phi(\mathscr{E})\mathscr{K}}$.
Clearly $\Phi(\mathscr{E})\mathscr{K}\subseteq \mathscr{K}'$. Due to
$\overline{\varphi(\mathscr{A}) \mathscr{K}} \subseteq \mathscr{K}$
we observe that $\overline{\varphi(\rho(x)^2) \mathscr{K}}\subseteq
\mathscr{K}$, whence $\overline{\Phi(x)^*\Phi(x) \mathscr{K}}
\subseteq \mathscr{K}$ for all $x\in \mathscr{E}$. Now let $x \neq
y$. In the notation of Theorem \ref{main} we have
\begin{eqnarray*}
\langle\Phi(x)^*\Phi(y)h, h'\rangle &=& \langle\Phi(y)h~,\Phi(x)h'\rangle=\langle\varphi(\rho(y))h_y, \varphi(\rho(x))h'_x\rangle\\
&=&\sum_{z\in \mathscr{E}}\langle\varphi(\rho(y)) h_y(z),
\varphi(\rho(x))h'_x(z)\rangle=0,
\end{eqnarray*}
for all $h,h'\in \mathscr{H}$. Put $h'=\Phi(x)^*\Phi(y)h$ to get
$\langle\Phi(x)^*\Phi(y)h, \Phi(x)^*\Phi(y)h\rangle=0$. So that
$\Phi(x)^*\Phi(y)h=0$. Therefore $\Phi(x)^*\Phi(y) \mathscr{K}=0
\mathscr{K}\subseteq \mathscr{K}$. It follows that
$\Phi(E)^*\overline{\Phi(E) \mathscr{K}} \subseteq
\overline{\Phi(E)^*\Phi(E) \mathscr{K}} \subseteq \mathscr{K}$.
 Since $\Phi$ is irreducible, we conclude that $(\mathscr{K},\mathscr{K}')=(0,0)$ or $(\mathscr{K},\mathscr{K}')=(\mathscr{H},\mathscr{H}')$, hence $\mathscr{K}=0$ or $\mathscr{K}=\mathscr{H}$. This implies that $\varphi$ is irreducible.

Now assume that $\varphi$ is irreducible. It follows from
\cite[Remark 4.1.4]{mo} that $\varphi$ is nondegenerate.
By Theorem \ref{5}, $\Phi$ is nondegenerate.\\
Consider $(\mathscr{K},\mathscr{K}')$ as a $\Phi$-invariant pair of
subspaces. Any $a\in \mathscr{A}$ can be represented as
$a={\displaystyle{\lim_{n\to \infty}}} \sum_{i=1}^{k_n}
\lambda_{i,n} \rho(x_{i,n})^2$ for some $k_n \in \mathbb{N}$,
$x_{i,n}\in \mathscr{E} $ and $\lambda_{i,n}\in \Bbb{C}$. Hence
\begin{eqnarray*}
\varphi(a) \mathscr{K}={\displaystyle{\lim_{n\to \infty}}}
\sum_{i=1}^{k_n} \lambda_{i,n}\varphi(\rho(x_{i,n}))^2
\mathscr{K}={\displaystyle{\lim_{n\to \infty}}} \sum_{i=1}^{k_n}
\lambda_{i,n}\Phi(x_{i,n})^*\Phi(x_{i,n}) \mathscr{K} \subseteq
\mathscr{K},
\end{eqnarray*}
Hence $\mathscr{K}=0$ or $\mathscr{K}=\mathscr{H}$.\\
If $\mathscr{K}=0$ then $\Phi(\mathscr{E})^*\mathscr{K}'\subseteq \mathscr{K}=0$, and for every $\xi'\in \mathscr{K}'$ we have $0=\langle\Phi(x)^*\xi', \xi\rangle=\langle\xi', \Phi(x)\xi\rangle$ for $x\in \mathscr{E}$ and $\xi\in \mathscr{H}$, so that $\mathscr{K}'\perp \overline{\Phi(\mathscr{E})\mathscr{H}}=\mathscr{H}'$. Since $\mathscr{K}'\subseteq \mathscr{H}'$, we have $\mathscr{K}'=0$.\\
If $\mathscr{K}=\mathscr{H}$, then
$\mathscr{H}'=\overline{\Phi(\mathscr{E})\mathscr{H}}=\overline{\Phi(\mathscr{E})\mathscr{K}}\subseteq
\mathscr{K}'$. Hence $\mathscr{K}'=\mathscr{H}'$. Therefore $\Phi$
is irreducible.\label{11}
\end{proof}

\begin{remark} The result may fail, if the condition of being full is dropped. The closed subspace $P(\mathscr H)$ in Remark \ref{7} when $\varphi: \mathscr A \to \mathbb B (\mathscr H)$ is irreducible provides a counterexample.
\end{remark}

Next we present some conditions under which the quasi-representation
$\Phi$ is nondegenerate and irreducible.
\begin{corollary} Let $\mathscr{E}$ be a full Finsler $\mathscr A$-module and let $\varphi: \mathscr{A}\to \mathbb{B}(\mathscr{H})$ is irreducible.
Then the quasi-representation $\Phi: \mathscr{E}\to
\mathbb{B}(\mathscr{H},\mathscr{K})$ is nondegenerate and
irreducible.
\end{corollary}

\begin{proof} Since $\varphi$ is irreducible, it is nondegenerate. Since $\mathscr{E}$ is full, by Theorem \ref{5}, $\Phi$ is nondegenerate and by Theorem \ref{6}, $\Phi$ is irreducible.
\end{proof}

\begin{theorem} Let $\mathscr{E}$ be a full Finsler $\mathscr{A}$-module. Then $\Phi(\mathscr{E})$ is a subset of the space $\mathbb{K}(\mathscr{H},\mathscr{H}')$ of all compact operators from $\mathscr{H}$ into $\mathscr{H}'$ if and only if $\varphi(\mathscr{A})\subseteq \mathbb{K}(\mathscr{H})$.
\end{theorem}
\begin{proof} Suppose $\varphi(\mathscr{A})\subseteq \mathbb{K}(\mathscr{H})$. Applying the Hewitt--Cohen factorization theorem we have $\Phi(\mathscr{E})=\Phi(\mathscr{E}\mathscr{A})=\Phi(\mathscr{E})\varphi(\mathscr{A}) \subseteq \mathbb{K}(\mathscr{H},\mathscr{H}')$.\\
Conversely, suppose that $\Phi(\mathscr{E})\subseteq
\mathbb{K}(\mathscr{H},\mathscr{H}')$. Since $\mathscr{E}$ is full
we have
\begin{eqnarray*}
\varphi(a)={\displaystyle{\lim_{n\to \infty}}} \sum_{i=1}^{k_n}
\lambda_{i,n}\varphi(\rho(x_{i,n}))^2 ={\displaystyle{\lim_{n\to
\infty}}} \sum_{i=1}^{k_n}
\lambda_{i,n}\Phi(x_{i,n})^*\Phi(x_{i,n})\in
\mathbb{K}(\mathscr{H}),
\end{eqnarray*}
where $a={\displaystyle{\lim_{n\to \infty}}} \sum_{i=1}^{k_n}
\lambda_{i,n} \rho(x_{i,n})^2$ for some $k_n \in \mathbb{N}$,
$x_{i,n}\in \mathscr{E} $ and $\lambda_{i,n}\in \Bbb{C}$.
\end{proof}

Now in the next two examples we illustrate the considered situations
in the notation of Theorem 2.3.

\begin{example} By \cite[Theorem 1.10.2]{d} the identity map $\varphi:
\mathbb{K}(\mathscr{H})\to \mathbb{B}(\mathscr{H})$ is irreducible.
It is known that the $C^*$-algebra $\mathbb{K}(\mathscr{H})$ is a
full Finsler module over $\mathbb{K}(\mathscr{H})$ with
$\rho(T)=\vert T\vert$. Hence the quasi-representation $\Phi:
\mathbb{K}(\mathscr{H})\to \mathbb{B}(\mathscr{H},\mathscr{K})$ is
nondegenerate and irreducible.
\end{example}

\begin{example}
Consider $\varphi=I: \mathbb{B}(\mathscr{H})\to
\mathbb{B}(\mathscr{H})$. Then
$\varphi(\mathbb{B}(\mathscr{H}))^c=\{T\in
\mathbb{B}(\mathscr{H})~;~\varphi(S)T=T\varphi(S), \mbox{ for all }
S\in \mathbb{B}(\mathscr{H})\}=\{T\in
\mathbb{B}(\mathscr{H})~;~ST=TS, \mbox{ for all } S\in
\mathbb{B}(\mathscr{H})\}=\Bbb{C}I$. Hence $\varphi$ is irreducible.
Also $\mathbb{B}(\mathscr{H})$ is a full Finsler
$\mathbb{B}(\mathscr{H})$-module, so that the quasi-representation
$\Phi: \mathbb{B}(\mathscr{H})\to
\mathbb{B}(\mathscr{H},\mathscr{K})$ is nondegenerate and
irreducible.
\end{example}

\textbf{Acknowledgement.} The authors would like to sincerely thank
Professor M. Joi\c{t}a for some useful comments improving the paper.

\end{document}